\documentclass[12pt,reqno]{amsart}

\usepackage{amsmath}
\usepackage{amssymb}
\usepackage{amstext}
\usepackage{amsxtra}
\usepackage{latexsym}
\usepackage{amsthm}
\usepackage{enumerate}
\usepackage{mathrsfs}
\usepackage{amsrefs}
\usepackage[all]{xy}
\usepackage[dvipdfm]{graphicx}
\usepackage{amsrefs}
\usepackage{wrapfig}
\usepackage{color}

\allowdisplaybreaks[1]

\newcommand{\C}{{\mathbb C}}
\newcommand{\F}{{\mathbb F}}

\newcommand{\Q}{{\mathbb Q}}

\newcommand{\Z}{{\mathbb Z}}

\newcommand{\Gam}{{\Gamma}}

\newcommand{\vep}{{\varepsilon}}

\newcommand{\sig}{{\sigma}}

\newcommand{\om}{{\omega}}

\newcommand{\sL}{\mathscr{L}}

\newcommand{\sO}{\mathscr{O}}

\DeclareMathOperator{\Coker}{Coker}
\DeclareMathOperator{\Hom}{Hom}

\DeclareMathOperator{\Ker}{Ker}

\DeclareMathOperator{\Aut}{Aut}

\DeclareMathOperator{\Spec}{Spec}

\DeclareMathOperator{\Pic}{Pic}

\DeclareMathOperator{\Gr}{Gr}
\DeclareMathOperator{\tor}{tor}

\DeclareMathOperator{\Kr}{Kr}
\DeclareMathOperator{\an}{an}

\newcommand{\lan}{\langle}
\newcommand{\ran}{\rangle}

\newcommand{\alg}{\textit{alg}}

\renewcommand{\tt}{{T}^{-1}}
\renewcommand{\div}{\operatorname{div}}

\theoremstyle{plain}
\newtheorem{thm}{Theorem}[section]
\newtheorem{lem}[thm]{Lemma}
\newtheorem{prop}[thm]{Proposition}

\theoremstyle{definition}

\newtheorem{remk}[thm]{Remark}

\makeatletter

\@addtoreset{equation}{section}
\makeatother

\begin{document}

%%%%%%%%%%%%%%%%%%%%%%%%%%%%%%%
%%    Article Information     %%
%%%%%%%%%%%%%%%%%%%%%%%%%%%%%%%
\title[]{Torsion points on hyperelliptic Jacobians
\\ 
via Anderson's $p$-adic soliton theory}
\author{Yuken~Miyasaka \and Takao~Yamazaki}
\address{Mathematical Institute, Tohoku University,
         Sendai 980-8578, Japan}
\curraddr{}
\email[Takao Yamazaki]{ytakao@math.tohoku.ac.jp}
\email[Yuken Miyasaka]{sa7m27@math.tohoku.ac.jp}
\urladdr{}
\dedicatory{}
\date{\today}
\translator{}
\subjclass[2010]{}
\thanks{The first author is supported by JSPS Research Fellowship for Young Scientists.
The second author is supported 
by Grant-in-Aid for Challenging Exploratory Research (22654001), Grant-in-Aid for Young Scientists (A) (22684001), and Inamori Foundation.}

%
%%%%%%%%%%%%%%%%%%%%%%%%%%%%%%%
%%%   ABST  %%
%%%%%%%%%%%%%%%%%%%%%%%%%%%%%%%
\begin{abstract}
We  show that
torsion points of certain orders are not on a theta divisor
in the Jacobian variety of a hyperelliptic curve
given by the equation $y^2=x^{2g+1}+x$ with $g \geq 2$.
The proof employs a method of Anderson 
who proved an analogous result for a cyclic quotient
of a Fermat curve of prime degree.
\end{abstract}
\maketitle

%\tableofcontents

%%%%%%%%%%%%%%%%%%%%%%%%%%%%%%%%
%%%%%%%%%%%%%%%%%%%%%%%%%%%%%%%%%
\section{Introduction}
Let $K$ be a field of characteristic zero.
Let $A$ be an abelian variety over $K$
and $Z$ $(\neq A)$ a closed subvariety of $A$.
A celebrated result of Raynaud \cite{Raynaud} 
implies that the intersection of $Z$ with
torsion points $A_{\tor}$ on $A$ is finite,
if $Z$ is a curve of genus at least two,
or if $A$ is absolutely simple.
However, it is usually not easy to determine this finite set
$Z \cap A_{\tor}$ explicitly for given $A$ and $Z$.

Now let us assume $A=J$ is the Jacobian variety of 
a smooth projective  geometrically 
connected curve $X$ of genus $g \geq 2$.
Of particular interest is 
the case where $Z=X$ 
is the Abel-Jacobi embedded image of $X$
with respect to some base point.
Since Coleman \cite{Coleman1}
started to study this problem,
many works have been done in this direction.
See \cite{Tzermias} for a lucid survey on this subject.
Anderson \cite{Anderson} considered
the case where $Z=\Theta$ is the theta divisor of $J$.
He proved that torsion points of 
certain prime orders are not on $\Theta$
when $X$ is a cyclic quotient of a Fermat curve of prime degree.
For details of this result and 
its generalization by Grant \cite{Grant},
see Remark \ref{rem:anderson} (2) below.
In order to prove his result,
Anderson developed a $p$-adic analogue of the theory of
{\it tau function},
which was originally introduced by Sato 
\cites{Sato, Sato-Sato} (see also \cite{Segal-Wilson})
in his study of soliton equations
(in the complex analytic setting).
In this paper, we apply Anderson's theory to other curves
and prove analogous results.

\subsection{Setting}\label{sect:setting}
To state our main result, we introduce notations.
Fix an integer $g \geq 2$.
Let $K$ be a field of characteristic zero 
that contains a primitive $4g$-th root $\zeta$ of unity.
We consider a hyperelliptic curve $X$ of genus $g$ over $K$
defined by the equation
\begin{equation}\label{eqn:hyperelliptic}
y^2=x^{2g+1}+x.
\end{equation}  
Let $\infty$ be the $K$-rational point 
at which the functions $x$ and $y$ have poles.
There is an automorphism $r$ of $X$ of order $4g$
defined by $r(x,y)=(\zeta^2x,-\zeta y)$.
Let $G := \lan r \ran$ be the subgroup of $\Aut(X)$ 
generated by $r$.
The Jacobian variety $J$ of $X$
will be considered as a $\Z[G]$-module
by the induced action of $G$.
(We will see in \S \ref{sect:otsubo}
that $J$ is absolutely simple when $g>45$.)
We define the theta divisor
$\Theta$ to be the set of $\sL\in J$ 
such that $H^0(X, \sL((g-1)\infty)) \not= \{ 0 \}$.
Note that $r(\infty)=\infty$ so that $\Theta$ is stable under the action of $r^*$.
For any $n \in \Z_{>0}$, we write
$J[n]$ for $n$-torsion subgroup of $J$.

\subsection{Main results}
Let $p$ be a prime number such that  
$p \equiv 1\mod 4g,$
and choose a prime ideal $\wp \subset \Z[\zeta]$ lying above $p$.
We write $\chi$ for the composition of 
\[ G \to \Z[\zeta]^* 
  \twoheadrightarrow (\Z[\zeta]/\wp)^* = \F_p^*
\]
where the first map is defined by $r \mapsto \zeta$.
We will show in Lemma \ref{lem:decomp} below
that we have
\[  \dim_{\F_p} J[p]^{\chi} = 1
\]
where $J[p]^{\chi}=\{ \sL \in J[p] ~|~ r^*\sL = \chi(r) \sL\}.$
Our main results are the following:

\begin{thm}\label{thm:maimtheorem}
We have 
\begin{equation*}
(J[p]^{\chi} + J[2] ) \cap \Theta \subseteq J[2].
\end{equation*}
\end{thm}

\begin{thm}\label{thm:maimtheorem2}
Assume that $K$ is a finite extension of $\Q_p$. 
Let $Q\in X(K)$ and put $\sL_Q := \sO_X(Q-\infty)$.
Assume that the coordinates 
$x(Q)$ and  $y(Q)$ of $Q$ belong to the integer ring of $K$.  
Then we have 
\begin{equation*}
(J[p]^{\chi} + \sL_Q ) \cap \Theta = \{\sL_Q\}.
\end{equation*}
\end{thm}

\begin{remk}\label{rem:anderson}
\begin{enumerate}
\item
The set $\Theta \cap J_{\tor}$ is explicitly determined 
when $g=2$  by Boxall-Grant \cite{Boxall-Grant}.
It consists of twenty-two points
(over an algebraically closed field).
\item
For the sake of comparison,
we recall Anderson's result \cite{Anderson}.
Fix an odd prime number $l$, integers $a \geq b>1$ such that $l+1=a+b$,
and a primitive $l$-th root $\zeta_l$ of unity. 
Let $X$ be the smooth projective curve defined by
$y^l=x^a(1-x)^{b}$,
and define $J$ and $\Theta$ similarly as above.
(By Koblitz-Rohrlich \cite{Kob-Roh},
$J$ is absolutely simple.)
There is an automorphism $\gamma$ of $X$
defined by $\gamma(x,y)=(x,\zeta_l y)$,
which induces a $\Z[\zeta_l]$-module structure on $J$
such that $\zeta_l$ acts by $\gamma^*$.
For an ideal $\frak{a}$ of $\Z[\zeta_l]$,
we write $J[\frak{a}]$ for the $\frak{a}$-torsion subgroup of $J$.
Let $p$ be a prime number such that $p\equiv1\mod l$
and take a prime ideal $\wp\subset\Z[\zeta_l]$ over $p$.
Anderson's result \cite[Theorem 1]{Anderson} is the following:
$$
(J[\wp] + J[(1-\zeta_l)])\cap\Theta \subseteq J[(1-\zeta_l)].
$$
Grant \cite{Grant} improved Anderson's result by showing
for all $n\ge1$
$$
(J[\wp^n] + J[(1-\zeta_l)])\cap\Theta \subseteq J[(1-\zeta_l)]
$$
under the assumption that $X$ is hyperelliptic 
(that happens iff $a \in \{(l+1)/2, l-1\}$).
\item
In our setting, $X, \infty, J$ and $\Theta$ are all defined over $\Q$,
and the choice of $\wp$ is arbitrary.
By taking different choices of $\wp$,
one can replace $J[p]^{\chi}$ 
by $J[p]^{\chi^i} := \{ \sL \in J[p] ~|~ r^* \sL = \chi(r)^i \sL \}$ 
for any $i \in (\Z/4g\Z)^*$
in Theorem \ref{thm:maimtheorem}.
(In our proof, though, the value of $s$ appearing
after \eqref{eq:lastentry} will be changed.
Note also that 
a similar statement does not hold for Theorem \ref{thm:maimtheorem2}
because $Q$ may not be defined over $\Q$.)
It is an open problem to extend this result
to $i$ which is not prime to $4g$.
Another open problem is to replace
$J[p]$ by $J[p^n]$ with $n>1$ in 
Theorems \ref{thm:maimtheorem}, \ref{thm:maimtheorem2}
(compare Grant's result recalled in (2) above).
\item
The crutial step in our proof 
where we need to assume $X$ to be a special curve
\eqref{eqn:hyperelliptic}
is in \S \ref{sect:auxlemma}.
It might be possible to apply our method to other curves.
See Remark \ref{rem:addedcomment} for more discussion
about the possibility and difficulity in it.
\end{enumerate}
\end{remk}

This paper is organized as follows.
In \S 2 
we recall some results from Anderson \cite{Anderson}.
In \S 3 we study geometry 
of the hyperelliptic curve \eqref{eqn:hyperelliptic}.
The proof of Theorems \ref{thm:maimtheorem} and \ref{thm:maimtheorem2}
is completed in \S 4.
The last section \S 5 is devoted to an illustration of 
Anderson's results recalled in \S2.

%%%%%%%%%%%%%%%%%%%%%%%%%%
%%%%%%%%%%%%%%%%%%%%%%%%%%
\section{Review of Anderson's theory}\label{sect:anderson}
In this section, 
we recall (bare minimum of)
results of Anderson \cite[\S 2, 3]{Anderson}.
We formulate all results
without any use of Sato Grassmannian
(which is actually central in Anderson's theory).
All results in this section are merely reformulation of loc. cit.,
but for the sake of completeness 
we include some explanation 
using Sato Grassmannian in \S \ref{sect:app}.

%%%%%%%%%%%%%%%%%%%%%%%
\subsection{Krichever pairs}\label{sect:kri}
Let $X$ be a smooth projective geometrically irreducible curve 
over a field $K$
equipped with a $K$-rational point $\infty$. 
We fix an isomorphism $N_0 : \hat{\sO}_{X,\infty} \cong K[[\tt]]$,
and write $N$ for the composition map
$\Spec K((\tt)) \to \Spec K[[\tt]] \overset{N_0}{\to} X$.
(Here $K[[\tt]]$ is the ring of power series
in $\tt$ with coefficients in $K$,
and $K((\tt))$ is its fraction field.)
An {\it $N$-trivialization} of a line bundle $\sL$ on $X$
is an isomorphism $\sig : N^*\sL\cong K((\tt))$
of $K((\tt))$-vector spaces
induced by an isomorphism
$\sig_0 : N_0^*\sL\cong K[[\tt]]$ of $K[[\tt]]$-modules. 
A pair $(\sL, \sig)$ of
a line bundle $\sL$ on $X$ and an $N$-trivialization $\sig$ of $\sL$
is called a {\it Krichever pair}.
Two Krichever pairs are said to be isomorphic
if there exists an isomorphism of line bundles
compatible with $N$-trivializations.
We write $\Kr(X, N)$ for the set of isomorphism
classes of Krichever pairs.
We have a canonical surjective map
\begin{equation*}
 [ \cdot ] : \Kr(X, N) \to \Pic(X), \qquad
 [(\sL, \sig)]  = \sL.
\end{equation*}
For each $n \in \Z$ we define 
$\Kr^n(X, N) := \{ (\sL, \sig) \in \Kr(X, N) ~|~ \deg(\sL)=n \}$
to be the inverse image of $\Pic^n(X)$ by $[ \cdot ]$.

\subsection{A Krichever pair associated to a Weil divisor}\label{sect:weildiv}
Let $D = \sum_{P \in X} n_P P$ 
be a Weil divisor on $X$.
The associated line bundle $\sO_X(D)$ admits an $N$-trivialization
$\sigma(D)$ induced by the composition
$\sO_X(D) \hookrightarrow K(X) \overset{N}{\to} K((\tt))
\overset{T^{-n_{\infty}}}{\to} K((\tt))$.
(Here $n_{\infty}$ is the coefficient of $\infty$ in $D$.)
Thus we obtain a Krichever pair $(\sO_X(D), \sigma(D))$.

\subsection{Vector space associated to a Krichever pair}\label{sect:vector}
For $(\sL,\sig) \in \Kr(X, N)$,
we define a $K$-subspace $W(\sL, \sig)$ of $K((\tt))$ by 
\begin{equation*}\label{eqn:Krichever_pair}
W(\sL, \sig) := 
\{\sig N^* f\in K((\tt))\ |\ f\in H^0(X\setminus\{\infty\}, \sL)\}.
\end{equation*}
Note that $A := W(\sO_X, N)$ is a $K$-subalgebra of $K((\tt))$
such that $\Spec A \cong X \setminus \{ \infty \}$,
and that $W(\sL, \sig)$ is an $A$-submodule of $K((\tt))$
for any $(\sL, \sig) \in \Kr(X, N)$.
The following fact is fundamental to us.
(See Proposition \ref{prop:corresp} for details.)
\begin{prop}\label{prop:fundamental}
Let $(\sL, \sig) ,(\sL', \sig') \in \Kr(X, N)$.
If $W(\sL, \sig)=W(\sL', \sig')$, then
we have $(\sL, \sig)=(\sL', \sig')$.
\end{prop}

\subsection{Admissible basis}\label{sect:kr-w}
Let $(\sL, \sig) \in \Kr(X, N)$.
Put $W=W(\sL, \sig)$ and $i_0 := \deg(\sL)+1-g$.
It follows from the Riemann-Roch theorem that
there is a $K$-basis $\{w_i\}_{i=1}^\infty$ of $W$ such that 
\begin{itemize}
\item[(1)] $\{\deg w_i\}_{i=1}^\infty$ is a strictly increasing sequence, 
\item[(2)] $w_i$ is monic for all $i$, and 
\item[(3)] $\deg (w_i-T^{i-i_0})$ is a bounded function of $i$.
\end{itemize}
(Here $\deg : K((\tt))^* \to \Z$ is 
the sign inversion of the normalized valuation,
and $w \in K((\tt))$ is called {\it monic} iff 
$\deg(w - T^{\deg w}) < \deg(w)$.)
Such a $K$-basis $\{w_i\}_{i=1}^\infty$ of $W$ will be called {\it admissible}.
We call $i(W):=i_0$ the {\it index} of $W$.
(The integer $i_0$ can be read off from $W$,
as it is the only integer that satisfies the property (3) above.)
The {\it partition} $\kappa=(\kappa_i)_{i=1}^\infty$ of $W$ 
is a non-increasing sequence of non-negative integers defined by
$$
\kappa_i := i- i(W) - \deg(w_i),
$$
which satisfies $\kappa_i=0$ for sufficiently large $i$. 
The partition $\kappa$ 
does not depend on a choice of an admissible basis.
(Actually, it depends only on $\sL$.)
The integer $\ell(\kappa):=\max\{i\ |\ \kappa_i\neq0\}$ will be called the {\it length} of 
the partition $\kappa$.
(See also comments after \eqref{eq:index}.)

\subsection{Group structure}\label{sect:groupstructure}
We regard $\Kr(X, N)$ as an abelian group by the tensor product,
so that the identity element is given by $(\sO_X, N)$.
Note that $[ \cdot ] : \Kr(X, N) \to \Pic(X)$ is a group homomorphism.
Take $(\sL, \sig), (\sL', \sig) \in \Kr(X, N)$
and let $(\sL'', \sig'')=(\sL \otimes \sL', \sig \otimes \sig')$
be their product.
Then
$W(\sL'', \sig'')$ coincides with the $K$-subspace of $K((\tt))$
spanned by 
$\{ ww' \in K((\tt)) ~|~ w \in W(\sL, \sig), w' \in W(\sL', \sig') \}$.

\subsection{Theta divisor}
Let us write $J:=\Pic^0(X)$ for the {\it Jacobian variety} of $X$.
Let us also write
$\Theta\subset J$ for the {\it theta divisor},
which is defined to be
the set of $\sL \in J$
such that $H^0(X, \sL((g-1)\infty)) \not= \{ 0 \}$.
Observe that $(\sL,\sig) \in \Kr^0(X, N)$
satisfies $\sL \in \Theta$ if and only if 
$$
 W(\sL, \sig) \cap T^{g-1}K[[\tt]] \neq\{0\},
$$
because there is an isomorphism
$
W(\sL, \sig) \cap T^{g-1}K[[\tt]] \cong H^0(X, \sL((g-1)\infty)).
$
(This is a key property which enables one
to interpret $\Theta$ as the `zero-locus' of the tau function.)

\subsection{Automorphism of a curve}\label{subsec:auto}
Suppose we are given 
two endomorphisms $r$ and $\bar{r}$ of $K$-schemes 
which fit in the commutative diagram
\begin{equation*}\label{eq:action}
\xymatrix{
 \Spec K((\tt)) \ar[r]^(0.68){N} \ar[d]_{\bar{r}} & X\ar[d]^{r} \\
 \Spec K((\tt)) \ar[r]^(0.68){N} & X.
}
\end{equation*}
In particular, it holds $r(\infty)=\infty$.
Then, for $(\sL, \sig) \in \Kr(X, N)$,
the composition
\[(r, \bar{r})^* \sig :  N^* r^* \sL \cong \bar{r}^* N^* \sL 
  \overset{\bar{r}^* \sigma}{\cong} 
  \bar{r}^* K((\tt)) = K((\tt))
\]
is an $N$-trivialization of $r^* \sL$.
(Here the last equality holds since
$\bar{r}$ induces 
an isomorphism $\bar{r}^* : K((\tt)) \to K((\tt))$ ).
Therefore we get an induced homomorphism
\[
  \Kr(X, N) \to \Kr(X, N), \qquad
  (\sL, \sig) \mapsto (r^* \sL, (r, \bar{r})^* \sig),
\]
which, by abuse of notation, will be denoted by $r^*$.
This homomorphism is compatible with $[ \cdot ]$
in the sense that $[r^*(\sL, \sig)]=r^* \sL$.

%%%%%%%%%%%%%%%%%%%%%%%%%%
\subsection{The $p$-adic analytic part of Krichever pairs}\label{sect:analy}
From now on, we assume
$p$ is a prime number and
$K$ is a finite extension of the field $\Q_p$ of $p$-adic numbers.
Let $|\cdot|$ the absolute value on $K$ such that $|p|=p^{-1}$.
Let $H(K)$ be the ring defined by 
\begin{equation*}
H(K):=\left\{ \sum_{i=-\infty}^{\infty}a_iT^i \ \biggm|\ 
a_i\in K,\ 
\sup_{i=-\infty}^{\infty}|a_i|<\infty,\ 
\lim_{i\to\infty}|a_i|=0  \right\}_.
\end{equation*}
Note that $H(K)$ is equipped with  the norm 
\begin{equation*}
\left\|\sum_i a_i T^i\right\| := \sup_{i}|a_i|,
\end{equation*}
and $(H(K), \| \cdot \|)$ is a $p$-adic Banach algebra over $K$ .

We write $\Kr_{\an}(X, N)$ for the subset of $\Kr(X, N)$
consisting of all Krichever pairs $(\sL, \sig)$
such that $W(\sL, \sig)$
admits an admissible basis $\{w_i\}$ satisfying
\begin{enumerate}
\item
$w_i \in H(K)$ for all $i$, and 
\item
$\|w_i\|=1$ for almost all $i$. 
\end{enumerate}
For each $n \in \Z$, we put 
$\Kr_{\an}^n(X, N) = \Kr_{\an}(X, N) \cap \Kr^n(X, N)$.

For $(\sL, \sig) \in \Kr_{\an}(X, N)$, 
we write $\bar{W}(\sL, \sig)$ for the closure of 
$W(\sL, \sig)$ in $H(K)$.
One recovers $W(\sL, \sig)$ from $\bar{W}(\sL, \sig)$
by $W(\sL, \sig) = \bar{W}(\sL, \sig) \cap K((\tt))$.
(Here we regard both $H(K)$ and $K((\tt))$ as
$K$-vector subspaces of $\prod_{i \in \Z} K T^i$.)
Hence the following proposition is a consequence of
Proposition \ref{prop:fundamental}.
\begin{prop}\label{prop:fundamental2}
Let $(\sL, \sig) ,(\sL', \sig') \in \Kr_{\an}(X, N)$.
If $\bar{W}(\sL, \sig)=\bar{W}(\sL', \sig')$, then
we have $(\sL, \sig)=(\sL', \sig')$.
\end{prop}

%%%%%%%%%%%%%%%%%%%%%%%%%%
\subsection{The $p$-adic loop group}
We define
the {\it $p$-adic loop group} $\Gam(K)$ 
to be the subgroup of $H(K)^\times$ 
consisting of all
$\sum_i h_iT^i \in H(K)^\times$ such that 
$|h_0|=1$,  
$|h_i|\le1$ for all $i\le0$, and 
there exists a real number $0<\rho<1$ such that 
\begin{equation*}
|h_i|\le \rho^i \quad \textrm{for\ all\ }i\ge1.
\end{equation*}
Define the subgroups $\Gam_+(K)$ and $\Gam_-(K)$ of $\Gam(K)$ by 
\begin{eqnarray*}
\Gam_+(K)&:=&\left\{\sum_i h_iT^i\in \Gam(K) \biggm| h_0=1,h_i=0 \ (i<0) \right\}_,\\
\Gam_-(K)&:=&\left\{\sum_i h_iT^i\in \Gam(K) \biggm| h_i=0 \ (i>0) \right\}_.
\end{eqnarray*}

\begin{prop}[{\cite[\S 3.3]{Anderson}}; see also \S \ref{sect:final} below]
\label{prop:loopgroupaction}
There is an action of $\Gam(K)$ on $\Kr_{\an}(X, N)$
characterized by the following property:
for any $h \in \Gam(K)$ and $(\sL, \sig) \in \Kr_{\an}(X, N)$,
we have $\bar{W}(h(\sL, \sig))=h\bar{W}(\sL, \sig)$.
(Here the right hand side means 
$\{ hw ~|~ w \in \bar{W}(\sL, \sig) \}$.)
Moreover, this action satisfies the following properties:
\begin{enumerate}
\item For any 
$h \in \Gam(K)$ and $(\sL, \sig) \in \Kr_{\an}(X, N)$,
we have $\deg [h(\sL, \sig)] = \deg [(\sL, \sig)]$.
\item For any 
$h \in \Gam_-(K)$ and $(\sL, \sig) \in \Kr_{\an}(X, N)$,
we have $[h(\sL, \sig)] = [(\sL, \sig)]$.
\item Suppose $(\sO_X, N) \in \Kr_{\an}(X, N)$.
For any $h \in \bar{W}(\sO_X, N) \cap \Gam(K)$ 
and $(\sL, \sig) \in \Kr_{\an}(X, N)$,
we have $[h(\sL, \sig)] = [(\sL, \sig)]$.
\end{enumerate}
\end{prop}

%%%%%%%%%%%%%%%%%%%%%%%%%
\subsection{Dwork loops and Anderson's theorem}
In his study of the $p$-adic properties of 
zeta functions of hypersurfaces over finite fields 
(see, for example, \cite{Dwork}),
Dwork constructed a special element of $\Gam(K)$
(which we call a Dwork loop).
We shall exploit his construction.
Assume that $K$ contains a $(p-1)$-st root $\pi$ of $-p$. 
Let $u$ be a unit of the integer ring of $K$.
A {\it Dwork loop} is defined by 
\begin{equation*}\label{eqn:Dwork} 
h(T):=\exp(\pi((uT)-(uT)^p)).
\end{equation*}
For all $i\ge0$, we have (see, for example \cite[Chapter I]{Koblitz}) 
$$
|h_i| \le | p |^{i(p-1)/p^2},
$$
where $h(T)=\sum_i h_iT^i$.
Therefore $h(T)\in \Gam_+(K)$.

The following theorem,
which is a consequence of a delicate analysis of 
Anderson's $p$-adic tau-function,
is technically crucial in \cite{Anderson}.
(See also \S \ref{sect:final}.)
\begin{thm}[{\cite[Lemma 3.5.1]{Anderson}}]
\label{thm:Dwork_loop}
Assume that  $p\ge7$.
Let $h$ be a Dwork loop
and $(\sL, \sigma) \in \Kr_{\an}^0(X, N)$.
We write $\kappa=(\kappa_i)_{i=1}^\infty$ and $\ell(\kappa)$
for the partition of $W(\sL, \sigma)$ 
and the length of $\kappa$. 
Assume further $W(\sL, \sigma)$ satisfies that 
\begin{itemize}
\item[(A1)] there exists an admissible basis 
$\{w_i\}_{i=1}^\infty$ such that 
$w_i \in H(K)$ and
$\|w_i\|=1$ for all $i\ge1$,
\item[(A2)] the partition 
$\kappa$ satisfies $\max\{\kappa_1,\ell(\kappa)\}<p/4$.
\end{itemize}
Then, we have
$W(h(\sL, \sigma))\cap T^{g-1}K[[\tt]]=\{0\}$.
Equivalently, we have
$$
[h(\sL, \sig)] \not\in \Theta.
$$ 
\end{thm}

\section{Geometry of a hyperelliptic curve}\label{sect:geometry}

%%%%%%%%%%%%%%%%%%%%%%%%%%%%%

In this section, 
we use the notations introduced in \S \ref{sect:setting}.

%\subsection{A hyperelliptic curve}
%Let $g>1$ be an integer.
%Let $K$ be a field of characteristic zero,
%and assume that $K$ contains a primitive $4g$-th root $\zeta$ of unity.
%Let $X$ be the hyperelliptic curve  given by the equation
%\eqref{eqn:hyperelliptic}:
%$$
%y^2=x^{2g+1}+x.
%$$
%There is an automorphism $r$ of $X$ of order $4g$
%given by
%$r(x,y)=(\zeta^2x,\zeta y)$.
%Let $G:=\lan r \ran \cong \Z/4g\Z$ 
%be the subgroup of $\Aut(X)$ generated by $r$.
%Note that $r(\infty)=\infty$,
%where $\infty \in X(K)$ is the point 
%at which the functions $x$ and $y$ have poles.

\subsection{Singular homology}
In this subsection we assume $K$ is a subfield of $\C$.
The singular homology $H_1(X(\C), \Z)$
is a free $\Z$-module of rank $2g$
on which $G$ acts linearly.
Let $\rho : G \to \Aut(H_1(X(\C), \Z))$
be the corresponding representation.
Let $\chi : G \to \mu_{4g}$ be the character
given by $\chi(r)=\zeta$.

\begin{lem}\label{lem:singularhom}
The representation $\rho \otimes \C$ is equivalent to
$\oplus_{i=1, 3, \cdots, 4g-1} \chi^i$.
In particular, 
the minimal polynomial of 
$\rho(r)$ is $F(X) := X^{2g}+1$.
\end{lem}
\begin{proof}
We consider a
$\C[G]$-module
$V=H^0(X, \Omega_{X/\C}^1) 
= \lan w_i =x^{i-1}dx/y ~|~ i=1, \cdots, g\ran_{\C}$.
A direct computation shows $r^*(w_i)=-\zeta^{2i-1}w_i$.
The lemma follows from an isomorphism
\[ H_1(X(\C), \Z) \otimes \C  \cong V \oplus \Hom(V, \C) \]
of $\C[G]$-modules.
\end{proof}

%%%%%%%%%%%%%%%%%%%%%%%%%%%%%
\subsection{Good trivialization}
The following is an easy consequence of Hensel's lemma:
\begin{lem}\label{lem:local_parameter}
There exists a unique element $u(T) \in 1+\tt\Z[[\tt]]$ such that 
\[ u(T)^{2g} - u(T)^{2g-1} + (\tt)^{4g}=0. \]
\end{lem}
We define two elements $x(T), y(T) \in \Z[[\tt]][T]$ by
$$
x(T) := T^2 u(T), \qquad 
y(T) := -Tx(T)^g.
$$
Note that
$x(T) \equiv T^2 \mod T\Z[[\tt]]$
and 
$y(T) \equiv -T^{2g+1} \mod T^{2g}\Z[[\tt]]$.
It follows from Lemma \ref{lem:local_parameter} that 
$(T^{-2} x(T))^{2g} - (T^{-2} x(T))^{2g-1} + (\tt)^{4g}=0$.
By multiplying $T^{4g}x(T)$, we get 
%$x(T)^{2g+1}- T^2 x(T)^{2g}+x(T) = 0$, namely
$$ 
y(T)^2 = x(T)^{2g+1} + x(T).
$$
Therefore we can define an injection
$K(X) \hookrightarrow  K((\tt))$
of $K$-algebras 
by associating $x$ and $y$ with $x(T)$ and $y(T)$ respectively.
This induces an isomorphism
$N_0 : \hat{\sO}_{X, \infty} \cong K[[\tt]]$,
and we can apply the results of \S \ref{sect:anderson}.
Note that $A := W(\sO_X, N)$ is the $K$-subalgebra of $K((\tt))$
generated by $x(T)$ and $y(T)$.

\subsection{Admissible basis of $A$}
\label{remk:A-basis}
We construct a $K$-basis $\{ w_i \}_{i=1}^\infty$ of $A$ such that 
\begin{enumerate}
\item $w_i \in \Z[[T^{-1}]][T]$ for all $i$,
\item $w_i - T^{2i-2} \in T^{2i-3}\Z[[T^{-1}]]$ for all $i \leq g+1$, and
\item $w_i - T^{i-1+g} \in T^{2g}\Z[[T^{-1}]]$ for all $i \geq g+2$.
\end{enumerate}
In particular,
$\{ w_i \}$ is admissible in the sense of \S \ref{sect:kr-w}.
First we put
$$
u_{i}= 
\left\{
\begin{array}{ll}
x(T)^{i-1} & (1 \le i \le g), \\
x(T)^{g+(i-g-1)/2} & (i >g,  ~i \not\equiv g \mod 2), \\
-y(T)x(T)^{(i-g-2)/2} & (i>g, ~i \equiv g \mod 2 ). \\
\end{array}
\right.  
$$
Note that $u_i \in \Z[[T^{-1}]][T]$ for all $i$
and $\{ u_i \}$ is a $K$-basis of $A$.
We set $w_i = u_i$ for $i \leq g+1$.
Suppose we have constructed $w_1, \cdots, w_{i-1}$ for some $i \geq g+2$.
There exists  $\delta \in \lan w_1, \cdots, w_{i-1} \ran_{\Z}$
such that $u_i - T^{i-1+g} - \delta \in T^{2g}\Z[[T^{-1}]]$.
We then set $w_i := u_i - \delta$.
Note that the partition of $A$ is 
$$
(g, g-1,\cdots , 2,1,0,0,\cdots),
$$
and its length is $g$.

%%%%%%%%%%%%%%%%%%%%%%%%%%%%%%
\subsection{Two-torsion points}\label{sect:twotorsion}
For any $\sL \in J[2]$, 
we shall construct an $N$-trivialization $\sig$ of $\sL$
such that $W(\sL, \sig)$ admits an admissible basis 
$\{w_i \}$ satisfying 
$w_i \in \Z[\zeta][[T^{-1}]][T]$ for all $i$.

Recall that the Weierstrass points on $X$ are
$$
\infty,\ P_0=(0,0),\ \textrm{and} \ P_i=(\zeta^{2i-1},0)\quad (1\le i\le 2g).  
$$
It is proved in \cite[Chapter III, \S2]{MumfordII} that 
the two-torsion subgroup $J[2]$ of $J$ 
consists of line bundles associated to Weil divisors
$$ 
D_I :=\sum_{i\in I}(P_i-\infty),\quad I\subset\{0,1,\cdots,2g\},\ |I|\le g. 
$$
For a subset  $I \subset\{0,1,\cdots,2g\}$ such that $s:=|I|\le g$,
we get a Krichever pair
$(\sL_I, \sig_I) := (\sO_X(D_I), \sig(D_I))$
by the construction in \S \ref{sect:weildiv}.
We further set $L_I := W(\sL_I, \sig_I)$.

We construct a basis $\{w_{I,i} \}_{i=1}^\infty$ of $L_I$ as follows: 
define an element $f_I$ of $H^0(X\setminus\{ \infty \}, \sL_I) \subset K(x,y)$ by
$$
f_I:=y\prod_{j\in I}(x-x(P_j))^{-1}.
$$
Note that the divisor of $f_I$ satisfies   
$$\div(f_I)=
 \sum_{j\not\in I}P_j -\sum_{j\in I}P_j - (2g-2s+1)\infty.
$$
Now we define 
for $1\le i\le g-s$,
$$ 
u_{I,i} := T^s x(T)^{i-1}
$$
and for $1\le i$,
$$
u_{I, g-s+i}= 
\left\{
\begin{array}{ll}
T^sx(T)^{g-s+(i-1)/2} & (i:\text{odd}) \\
T^sf_I(T)x(T)^{(i-2)/2} & (i:\text{even}), \\
\end{array}
\right.  
$$
where 
$f_I(T)$ is the image of $f_I$ by the embedding $N^*: K(x,y) \hookrightarrow K((\tt))$. 
One sees that 
$$
\deg(u_{I, i})= 
\left\{
\begin{array}{ll}
2i-2+s& (1\le i \le g-s) \\
i+g-1 & (g-s <i ). \\
\end{array}
\right.  
$$
Therefore $\{u_{I,i} \}_{i=1}^\infty$ is a $K$-basis of $L_I$
such that $u_{I, i} \in \Z[\zeta][[T^{-1}]][T]$ for all $i$.
Now we can produce an admissible  basis  $\{ w_{I, i} \}$ of $L_I$ 
with required properties
by the same procedure as \S \ref{remk:A-basis}.
Note that  the partition of $L_I$ is 
\begin{equation*}
(g-s, g-s-1, \cdots, 2, 1, 0 , 0, \cdots),
\end{equation*}
and the length of the partition is $g-s$.

%%%%%%%%%%%%%%%%%%%%%%%%%%%%%
\subsection{Points of degree one}\label{sect:degreeone}
We fix a non-Weierstrass point $Q\in X(K)$. 
Let $(\sL_Q, \sig_Q)$ be the Krichever pair
associated to the Weil divisor $Q-\infty$ 
under the construction in \S \ref{sect:weildiv}.
We are going to construct 
an admissible basis  $\{w_{Q, i} \}$ 
of $L_Q := W(\sL_Q, \sig_Q)$
satisfying $w_{Q, i} \in \Z[x(Q), y(Q)][[T^{-1}]][T]$ for all $i$.

We define a function $f_Q \in
H^0(X\setminus \{ \infty \}, \sL_Q)\subset K(x,y)$:
$$
f_Q:=l_Q\cdot(x-x(Q))^{-1}, \qquad
l_{Q}:= y-x+y(Q)+x(Q).
$$
A straightforward computation shows that
$\div(f_Q) + Q+(2g-1)\infty$
is an effective divisor of degree $2g$.
We construct a basis $\{u_{Q,i} \}_{i=1}^\infty$ of $L_Q$ as follows: 
for $1\le i\le g$,
$$ 
u_{Q,i} := Tx(T)^{i-1}
$$
for $1\le i$,
$$
u_{Q, g+i}:= 
\left\{
\begin{array}{ll}
Tf_Q(T)x(T)^{(i-1)/2} & (i:\text{odd}) \\
Tx(T)^{g+(i-2)/2} & (i:\text{even} ), \\
\end{array}
\right.  
$$
where 
$f_Q(T)$ 
is the image of $f_Q$ in $K((\tt))$ by the embedding $N^*$. 
Note that 
$f_Q(T)$ belongs to $\Z[x(Q), y(Q)][[T^{-1}]][T]$,
hence so does $u_{Q, i}(T)$.
(Here we used a fact that an element 
$\sum_{i=-\infty}^{n} c_i T^i \in \Z[x(Q), y(Q)][[T^{-1}]][T]$
with $c_n \not= 0$ is invertible
if and only if $c_n \in \Z[x(Q), y(Q)]^*$.)
One sees that $$
\deg(u_{Q, i})= 
\left\{
\begin{array}{ll}
2i-1 & (1\le i \le g) \\
i+g-1 & (g<i ). \\
\end{array}
\right.  
$$
Therefore $\{u_{Q,i} \}_{i=1}^\infty$ is a $K$-basis of $L_Q$
such that $u_{Q, i} \in \Z[x(Q), y(Q)][[\tt]][T]$ for all $i$.
Now we can produce an admissible  basis  $\{ w_{Q, i} \}$ of $L_Q$ 
with required properties
by the same procedure as \S \ref{remk:A-basis}.
Note that  the partition of $L_Q$ is 
\begin{equation*}
(g-1, g-2, \cdots, 1, 0 , 0, \cdots),
\end{equation*}
and its length is $g-1$.

\subsection{Action of $G$ on $\Kr(X, N)$}\label{sect:actionofg}
We define a $K$-algebra automorphism $\bar{r}$
on $K((\tt))$ by
\begin{equation*}
\bar{r}\left(\sum_i a_iT^i\right):=\sum_i a_i(\zeta T)^i.
\end{equation*}
Then the diagram
\[
\xymatrix{
 \Spec K((\tt)) \ar[r]^(0.68){N} \ar[d]_{\bar{r}} & X\ar[d]^{r} \\
 \Spec K((\tt)) \ar[r]^(0.68){N} & X.
}
\]
commutes.
By \S \ref{subsec:auto},
we get an induced action of $G$ on $\Kr(X, N)$.
It holds that $W(r(\sL, \sigma)) = \bar{r}(W(\sL, \sig))
(:= \{ \bar{r}(w) ~|~ w \in W(\sL, \sig) \}$).

\subsection{Remark on the simplicity of Jacobian}\label{sect:otsubo}
{\footnote{This remark is communicated to us by Noriyuki Otsubo.}}
(The result of this subsection will not be used in the sequel.)
We suppose $K$ is an algebraically closed field.
We deduce from
a result of Aoki \cite{Aoki} that the Jacobian variety of $X$ 
is simple as an abelian variety,
at least when $g>45$.
To see this,
let $X'$ be a smooth projective curve over $K$
defined by $s^{4g}=t(1-t)$.
Note that the curve $X'$ is a quotient of the Fermat curve of degree $4g$.
There exists a degree two map
$\pi : X' \to X$
 given by $x=c^2s^2, ~y=c(2t-1)s$,
where $c=(-4)^{1/4g}$.
Aoki's result \cite{Aoki} shows that
the Jacobian variety of $X'$
has exactly two simple factors, provided $g>45$.
The existence of $\pi$ shows that 
the Jacobian variety of $X$ must be one of two simple factors.

\section{Proof of main theorem}
We keep the notation and assumption in \S \ref{sect:geometry}.
Let $p$ be a prime number such that 
$$
p \equiv 1 \mod 4g.
$$
Let $\wp$ be a prime ideal of $\Z[\zeta]$ lying above $p$. 
 Since the hyperelliptic curve \eqref{eqn:hyperelliptic} is defined over $\Q(\zeta)$, 
 we may assume that $K$ is a finite extension of $\Q_p$ such that $\wp = \Z[\zeta] \cap p\Z_p$ in $K$.
We further assume that $K$
contains all elements of $J[p]$ and 
$(p-1)$-st roots of all rational integers. 

\subsection{$p$-torsion of the Jacobian}
Note that $\F_p$ contains all the $4g$-th roots of unity.
Put $\bar\zeta := \zeta \mod\wp \in \F_p$. 
Choosing an embedding $\bar{\Q}_p \hookrightarrow \C$,
we get an isomorphism $J[p] \cong H_1(X(\C), \Z) \otimes \F_p$.
The representation $\rho_p : G \to \Aut(J[p])$ 
is thus equivalent to $\rho \otimes \F_p$.
Therefore Lemma \ref{lem:singularhom} implies the following:

\begin{lem}\label{lem:decomp}
The minimal polynomial of 
$\rho_p(r)$ is 
$$
F(X)\mod p\  = \prod_{i=1, 3, \cdots, 4g-1} (X-\bar{\zeta}^i).$$
Consequently, we have
\[  J[p] = \bigoplus_{i=1, 3, \cdots, 4g-1} J[p]^{\chi^i}, 
\quad
\dim_{\F_p} J[p]^{\chi^i}=1 ~~(i=1, 3, \cdots, 4g-1).
\]
Here, by abuse of notation,
we write $\chi^i$ for the composition 
$G \overset{\chi^i}{\to} \mu_{4g} \hookrightarrow 
\Z_p^* \overset{\mod p}{\twoheadrightarrow} \F_p^*$.
\end{lem}

\subsection{An auxiliary lemma}\label{sect:auxlemma}

The following lemma plays an important role in our proof for constructing $p$-torsion points.
This is the crucial point
where we need to assume $X$ to be a special curve
given by the equation \eqref{eqn:hyperelliptic}.
See Remark \ref{rem:addedcomment} below.

\begin{lem}\label{lem:aux}
We have an equation
\begin{equation}\label{eqn:T^p}
T^p - e_0 T = a(T) + g(T)
\end{equation}
for some $e_0 \in \Z_{(p)}^*$, 
$a(T)\in A \cap \Z[[\tt]][T]$ and $g(T)\in\tt \Z[[\tt]]$.
\end{lem}
\begin{proof}
Setting $p=4gp'+1$,
we write
$$ 
x^{2gp'} (1+x^{-2g})^{2gp'} = e_+(x) + e_0 + e_-(x) 
$$
where 
$e_{\pm}(x) \in x^{\pm 2g}\Z[x^{\pm 2g}]$, respectively,  and $e_0 \in \Z$.
Note that 
$e_0 = 
\begin{pmatrix} 2gp' \\ p' \end{pmatrix}$
is a $p$-adic unit.
We compute
\begin{align*}
e_+(x) + e_0 + e_-(x) &= x^{2gp'} (1+x^{-2g})^{2gp'}
 = (x+x^{1-2g})^{2gp'}
\\
 &= \left(\frac{x^{2g+1}+x}{x^{2g}}\right)^{2gp'}
 = \left(\frac{y^2}{x^{2g}}\right)^{2gp'}
= \left(\frac{-y}{x^{g}}\right)^{p-1}.
\end{align*}
Recalling $y(T)=-Tx(T)^g$,
we get an equation in $K((\tt))$
\begin{equation*}
T^p - e_0 T = a(T) + g(T)
\end{equation*}
where 
$a(T) := -y(T) {e_+(x(T))}/{x(T)^g}$ 
and $g(T) := T e_-(x(T))$.
Observe that $a(T)$ is in the image of $A=K[x, y]$ in $K((\tt))$
(since $e_+(x) \in x^{2g}\Z[x]$) 
and that $g(T) \in \tt \Z[[\tt]]$.
\end{proof}

\begin{remk}\label{rem:addedcomment}
\footnote{This remark is communicated to us by Shinichi Kobayashi.}
If one does not care much about integrality of the coefficients,
the decomposition \eqref{eqn:T^p}
holds under weaker assumptions.
To see this, 
using the notation in \S 2,
we consider a direct sum decomposition
\begin{equation}\label{eq:directsum}
 K((\tt)) =
 A \oplus K[[\tt]] \tt \oplus (\bigoplus_{i=1}^g K T^{w_i}), 
\end{equation}
where $w_1=1 < w_2 < \cdots <w_g<2g$ is the
{\it Weierstrass gap sequence}.
Thus we can write 
$T^p = a(T) + g(T) + \sum_{i=1}^g e_{i-1} T^{w_i}$
with $a(T) \in A, ~g(T) \in K[[\tt]] \tt$ 
and $e_0, \cdots, e_{g-1} \in K$.
Suppose that the automorphism $\bar{r}$ in \S 2.7
satisfies $\bar{r}(T) = \zeta T$ for a
primitive $n$-th root of unity $\zeta$
such that $p \equiv 1 \mod n$ and $n \geq 2g$.
Then, 
since the decomposition \eqref{eq:directsum} is
preserved by the action of $\bar{r}$,
one has $e_1=\cdots=e_{g-1}=0$
and $T^p = a(T) + g(T) + e_0 T$.

However, in order to prove 
that $e_0$ is a $p$-adic unit
(which is important for our purpose),
we had to proceed by concrete construction
given above.
It seems to be an interesting problem to find
 a general method to detect if $e_0$ is a unit.
We hope to come back to this point in future work.
(It is also important that
the coefficients of $a(T)$ and $g(T)$ are
$p$-adically integral.)
\end{remk}

%%%%%%%%%%%%%%%%%%%%%%%%%%%%
\subsection{Decomposition of a Dwork loop}\label{sect:decdwork}
The result of \S \ref{remk:A-basis}
shows that $(\sO_X, N) \in \Kr_{\an}(X, N)$.
Recall that $\bar{A}:=\bar{W}(\sO_X, N)$
is the closure of $A=W(\sO_X, N)$ in $H(K)$.
Let $\pi$ and $\vep_0$ be $(p-1)$-st roots  
of $-p$ and $1/e_0$ respectively,
where $e_0 \in \Z_{(p)}^*$ is the number appearing in Lemma \ref{lem:aux}.
(They belong to $K$ by the assumption 
made at the beginning of this section.)
We define a Dwork loop
\begin{align*}
h_D(T):=& 
\exp(\pi((\vep_0T)-(\vep_0T)^p))\\
=&\exp( -\pi \vep_0^p(T^p-e_0T)).
\end{align*}
We write $\omega : \F_p^* \to \mu_{p-1} \subset \Z_p^*$
for the Teichm\"uller character
so that $\omega(i) \equiv i \mod p$.
For $i \in \Z$, we set $\omega(i) = \omega(i \mod p)$.
If we replace $\vep_0$ by $\omega(i) \vep_0$
for some $i \in \Z$,
then $h_D(T)$ will be changed to another Dwork loop
$h_D(\omega(i) T) \in \Gamma_+(K)$.

\begin{prop}\label{prop:comp_loop}
\begin{enumerate}
\item
There exist 
$h_A \in \bar{A}\cap \Gam(K)$ 
and $h_- \in\Gam_-(K)$ 
such that
$$
h_{D}(T)^p=h_A(T) h_-(T).
$$
\item Let $i \in \Z$.
There exist 
$h_{A, i}\in \bar{A}\cap \Gam(K)$ 
and $h_{-, i}\in\Gam_-(K)$ 
such that
$$
h_{D}(\omega(i)T) h_{D}(T)^{-i}=h_{A, i}(T) h_{-, i}(T).
$$

\end{enumerate}
\end{prop}

\begin{proof}
From the equation \eqref{eqn:T^p}, we have
\begin{align*}
h_D(T)^p=& 
\exp(-p\pi\vep_0^p(T^p-e_0T))\\
=&\exp(-p\pi\vep_0^pa(T))\cdot\exp(-p\pi\vep_0^p g(T)).
\end{align*}
Since $a(T)\in A \cap \Z[[\tt]][T]$ and $g(T)\in\tt \Z[[\tt]]$, we have
$$\begin{array}{l}
h_A(T):=\exp(-p\pi\vep_0^pa(T))\in \bar{A}\cap\Gam(K)\\
h_-(T):=\exp(-p\pi\vep_0^p g(T))\in\Gam_-(K),
\end{array}
$$ 
because the radius of convergence of $\exp(T)$ is $|p|^{1/(p-1)}=|\pi|$.
The first claim is proved.

Using the equation \eqref{eqn:T^p} 
and $\omega(i)^p=\omega(i)$, 
we compute
\begin{align*}
h_D(\omega(i) T)h_D(T)^{-i}=&\exp(-(\omega(i)-i)\pi \vep_0^p(T^p-e_0T))\\
=&\exp(-(\omega(i)-i)\pi\vep_0^pa(T))\cdot\exp(-(\omega(i)-i)\pi\vep_0^p g(T)).
\end{align*} 
Since $\omega(i)-i\equiv0\mod\wp$, we have
$$
\begin{array}{l}
h_{A, i}(T):=\exp(-(\omega(i)-i)\pi\vep_0^pa(T))\in \bar{A}\cap\Gam(K)\\
h_{-, i}(T):=\exp(-(\omega(i)-i)\pi\vep_0^p g(T))\in \Gam_-(K),
\end{array}
$$
and we are done.
\end{proof}

%%%%%%%%%%%%%%%%%%%%%%%%%%%%
\subsection{Construction of $p$-torsion elements}\label{sect:torsion}
Recall that we have constructed a Dwork loop $h_D(T)  \in \Gamma_+(K)$ 
in \S \ref{sect:decdwork}.
Recall also that we have defined an automorphism
$\bar{r}$ of $H(K)$ in \S \ref{sect:actionofg}
by $\bar{r}(h(T))=h(\zeta T)$.

\begin{prop}\label{prop:p-tor_Gr}
\begin{enumerate}
\item We have
$
[h_D(T) (\sO_X, N)] \in J \setminus \Theta.
$
\item We have
$\{ [h_D(\xi T) (\sO_X, N)] ~|~ \xi \in \mu_{p-1} \}
  = J[p]^{\chi} \setminus \{ 0 \}.$
\end{enumerate}
\end{prop}

\begin{proof}
(1)
Put $(\sL, \sig) := h_D(\sO_X, N)  \in \Kr(X, N)$.
By Proposition \ref{prop:loopgroupaction} (1),
we have $\deg(\sL)=0$.
The result of \S \ref{remk:A-basis} shows that
$(\sO_X, N) \in \Kr_{\an}^0(X, N)$
satisfies the assumptions (A1) and (A2)
of Theorem \ref{thm:Dwork_loop}.
It follows that $\sL \not\in \Theta$.

(2) 
We first show that $\sL \in J[p] \setminus \{ 0 \}$.
Note that (1) implies that $\sL \not= 0$.
For $K$-subspaces $V_1, \cdots, V_m$ of $H(K)$,
we write $V_1 \cdot \ldots \cdot V_m$ for the 
$K$-span of $\{ \prod_{j=1}^m u_j ~|~ u_j \in V_j \}$.
When $V=V_1=\cdots=V_m$ we write
$V^m := V \cdot \ldots \cdot V$.
Let $V=\bar{W}(\sL, \sig)$.
Proposition \ref{prop:loopgroupaction} shows that $V=h_D\bar{A}$.
Thus $V^p = h_D^p\bar{A}$.
By Proposition \ref{prop:fundamental2} 
and \S \ref{sect:groupstructure}, 
we have 
$(\sL, \sig)^{\otimes p} = h_D^p(\sO_X, N)$.
Propositions \ref{prop:comp_loop} (1)
and \ref{prop:loopgroupaction}
show $[h_D^p(\sO_X, N)]=[(\sO_X, N)]$.
We conclude $\sL^{\otimes p}=\sO_X$.

Similarly, 
Proposition \ref{prop:comp_loop} (2)
shows that for all $i \in \Z$
\[
[h_D(\om(i)T)h_D(T)^{-i}(\sO_X,N)]=[(\sO_X,N)],
\]
thus we have 
\begin{equation}\label{eq:lastentry}
[h_D(\omega(i) T) (\sO_X, N)]=[h_D(T)^i(\sO_X,N)]=\sL^{\otimes i}.
\end{equation}
In particular,
if we take $s \in \Z$ such that $\omega(s)=\zeta (=\chi(r))$,
we get
\[
r^*(\sL) 
=[\bar{r}^*(h_D(T))(\sO_X, N)]
=[h_D(\zeta T)(\sO_X, N)]
\overset{}{=}\sL^{\otimes s}
=\chi(r) \sL,
\]
This shows $\sL \in J[p]^{\chi}$
and hence $J[p]^{\chi}$ is a cyclic group of
order $p$ generated by $\sL$.
Now \eqref{eq:lastentry} completes the proof.
\end{proof}

%%%%%%%%%%%%%%%%%%%%%%%%%%%%%
\subsection{Proof of Theorem \ref{thm:maimtheorem}}
We may suppose $K$ is a finite extension of $\Q_p$
satisfying the conditions stated
at the beginning of this section.
%By Remark \ref{remk:choice_wp},
%we may also assume $i=1$.
Take $\sL \in J[2]$ and $\sL' \in J[p]^{\chi} \setminus \{ 0 \}$.
We need to show $\sL \otimes \sL' \not\in \Theta$.
By Proposition \ref{prop:p-tor_Gr},
there exists a Dwork loop $h$ such that
$\sL' = [h(\sO_X, N)]$.
By \S \ref{sect:twotorsion},
there exists an $N$-trivialization $\sig$ of $\sL$
such that $W(\sL, \sig)$ admits an admissible basis 
$\{w_i \}$ satisfying $w_i \in \Z[\zeta][[T^{-1}]][T]$ for all $i$.
Hence $(\sL, \sig)$ belongs to $\Kr_{\an}^0(X,N)$ and  satisfies the assumptions
(A1) and (A2) of Theorem \ref{thm:Dwork_loop}.
It follows that $[h(\sL, \sig)] \not\in \Theta$.
By 
Propositions \ref{prop:fundamental2},
\ref{prop:loopgroupaction} and \S \ref{sect:groupstructure},
we have 
$[h(\sL, \sig)] =  [(\sL, \sig)]\otimes[h(\sO_X, N)]  = \sL \otimes \sL'$.

\subsection{Proof of Theorem \ref{thm:maimtheorem2}}
We may assume $Q$ is a non-Weierstrass point
by Theorem \ref{thm:maimtheorem}.
Then the same proof as the previous subsection works
if we put \S \ref{sect:degreeone}
in the place of \S \ref{sect:twotorsion}.

%%%%%%%%%%%%%%%%%%%%%%%%%%
%%%%%%%%%%%%%%%%%%%%%%%%%%
\section{Appendix: Sato Grassmannian}\label{sect:app}

In this section,
we explain Anderson's theory \cite{Anderson} 
in a style much closer to his original framework.
It will be apparent that
the results in \S \ref{sect:anderson}
are the same results stated in another way.

%%%%%%%%%%%%%%%%%%%%%%%%%%%

\subsection{Sato Grassmannian}
We work under the notation and assumption 
in \S \ref{sect:kri}.
The {\it Sato Grassmannian} $\Gr^\alg(K)$ is the set of 
all $K$-subspace $V\subset K((\tt))$ such that 
the $K$-dimensions of the kernel and cokernel of the map 
\begin{equation*}
f_V : V\to K((\tt))/K[[\tt]]\ ;\ v\mapsto v+ K[[\tt]]
\end{equation*}
are finite. 
The {\it index} of $V\in\Gr^\alg(K)$  is defined by
\begin{equation}\label{eq:index}
i(V):=\dim_K \Ker(f_V)- \dim_K \Coker(f_V).
\end{equation}
(The fibers of the map $i : \Gr^{\alg}(K) \to \Z$
are considered as `connected components' of $\Gr^{\alg}(K)$,
and each connected component
admits a {\it Schubert cell decomposition} 
indexed by the set of all partitions,
but we do not need these facts.)

Recall that $A := W(\sO_X, N)$ is a $K$-subalgebra of $K((\tt))$.
For $V \in \Gr^{\alg}(K)$,
we set $A_V := \{ f \in K((\tt)) ~|~ fV \subset V \}$,
which is a $K$-subalgebra of $K((\tt))$.
We define 
\[  \Gr_A^\alg(K) := \{ V \in \Gr^{\alg}(K) ~|~ A_V = A \}. \]
For $V, V' \in \Gr_A^{\alg}(K)$, we define their product to be 
$V \cdot V' = \lan w w' ~|~ w \in V, w' \in V' \ran_K$,
under which $\Gr_A^{\alg}(K)$ becomes an abelian group.

\begin{prop}[{\cite[\S 2.3]{Anderson}; see also \cite{Mumford2}}]
\label{prop:corresp}
The construction of \S \ref{sect:vector}
defines an isomorphism of abelian groups
\[ W : \Kr(X, N) \to \Gr_A^{\alg}(K);
\qquad (\sL, \sig) \mapsto W(\sL, \sig)
\]
which satisfies the following properties:
\begin{enumerate}
\item 
We have $i(W(\sL, \sig))=\deg(\sL)+1-g$
for any $(\sL, \sig) \in \Kr(X, N)$.
\item 
For $V, V' \in \Gr^\alg_A(K)$,
one has $[W^{-1}(V)]=[W^{-1}(V')]$
if and only if $V=uV'$ for some $u \in K[[\tt]]^*$.
\end{enumerate}
\end{prop} 

All results in \S \ref{sect:kri}-\ref{sect:groupstructure}
are explained by this proposition.

%%%%%%%%%%%%%%%%%%%%%%%%%%
\subsection{$p$-adic Sato Grassmannian}
Now we use the assumption and notation of \S \ref{sect:analy}.
Let $H_+(K)$ and $H_{-}(K)$ be the closed $K$-subspaces of  $H(K)$ defined by
\begin{align*}
H_+(K)&:=\left\{\sum_i a_i T^i\in H(K) \ \biggm|\ a_i=0\ (\textrm{for\ all\ }i\le0)\right\}_,\\ 
H_{-}(K)&:=\left\{\sum_i a_i T^i\in H(K) \ \biggm|\ a_i=0\ (\textrm{for\ all\ }i>0)\right\}_.
\end{align*}

The {\it $p$-adic Grassmannian } $\Gr^{\an}(K)$ is the set of 
all $K$-subspaces $\bar V\subset H(K)$
such that 
$\bar V$ is the image of a $K$-linear injective map $w : H_+(K) \to H(K)$
satisfying the following conditions:
there exist $i_0 \in \Z$,
a $K$-linear operator $v : H_+(K) \to H_-(K)$ with $\|v \|\le1$,
and
a $K$-linear endomorphism $u$ on $H_+(K)$ with $\|u\|\le1$
that is a uniform limit of bounded $K$-linear operators of finite rank
(i.e. {\it completely continuous}),
such that
the map $T^{i_0}w$ has the form 
\begin{equation*}
T^{i_0}w=
\left[ 
\begin{array}{c}
1+u \\
v \\
\end{array}
\right]
: H_+(K) \to 
\left[ 
\begin{array}{c}
H_+(K) \\
H_-(K)\\
\end{array}
\right].
\end{equation*} 
The {\it index} of $\bar V\in\Gr^{\an}(K)$, 
denoted by $i(\bar V)$, is defined by the difference of the dimensions of 
the kernel and cokernel of the projection map $\bar V\to H_+(K)$.

\begin{prop}[{\cite[\S 3.2]{Anderson}}]\label{prop:W^alg}
There is an injective map
\begin{equation*}
\Gr^{\an}(K) \hookrightarrow  \Gr^\alg(K), \qquad
 \bar V \mapsto \bar V^\alg:=\bar V \cap K((\tt)).
\end{equation*}
For any $\bar V \in \Gr^{\an}(K)$, one has $i(\bar V)=i(\bar V^{\alg})$.
For $V\in\Gr^\alg(K)$, 
there exists $\bar V\in\Gr^{\an}(K)$ such that $\bar V^\alg=V$ 
if and only if 
$V$ has  an admissible basis $\{w_i\}$ such that 
$w_i \in H(K)$ for all $i$ and $\|w_i\|=1$ for almost all $i$. 
\end{prop}
By this proposition, 
we regard $\Gr^{\an}(K)$ as a subset of $\Gr^{\alg}(K)$.
It follows that 
$\Kr_{\an}(X, N) = \{ (\sL, \sig) \in \Kr(X, N) ~|~
W(\sL, \sig) \in \Gr^{\an}(K) \}$.

\subsection{Action of $p$-adic loop group and Anderson's theorem}
\label{sect:final}
In \cite[\S 3.3]{Anderson},
the action 
\[ \Gam(K) \times \Gr^{\an}(K) \to \Gr^{\an}(K), \qquad 
 (h, \bar V) \mapsto h \bar V := \{ hv ~|~ v \in \bar V \} 
\]
of $\Gam(K)$ on $\Gr^{\an}(K)$
is defined. 
Proposition \ref{prop:loopgroupaction} 
is also proved in loc. cit.

Finally, Theorem \ref{thm:Dwork_loop}
is a reformulation of \cite[Lemma 3.5.1]{Anderson}.
Anderson proved this extraordinary result
by introducing the $p$-adic version of {\it Sato tau-function},
which plays a central role in
Sato's theory of KP hierarchy
(see \cite{Sato, Sato-Sato, Segal-Wilson}). 
Anderson's proof of Theorem \ref{thm:Dwork_loop} is based on
a careful estimate of the tau function.

\vspace{5mm}

\noindent
{\it Acknowledgement.} 
We would like to express our gratitude to Noriyuki Otsubo
and Shinichi Kobayashi
for his insightful comments.
In particular, the remarks in \S \ref{sect:otsubo} 
and \ref{rem:addedcomment}
are suggested by them.
We are also deeply grateful to Takeshi Ikeda
for stimulating discussion.
We learned the importance of 
the equation of the form \eqref{eqn:T^p} from him.
%Thanks are also due to the referees
%for pointing out mistakes.
%The proofs in \S \ref{sect:decdwork} and \ref{sect:torsion}
%are also clarified by their comments.

\bibliographystyle{plain}
%\bibliography{references.bib} 


%%%%%%%%%%%%%%%%%%%%%%%%%%%%%%%%%%%%%%%%%%%%%%%%%%%%%%%%%%%%%%%%%%%%%%%%%%%%%%%%%%%%%%%%%%%%%%%%%%%%%%
\end{document}